\documentclass[twoside,11pt,reqno]{amsart}
\usepackage{amsmath,amsthm,amssymb,amstext,amsfonts,amscd}
\usepackage{graphicx}
\usepackage{multirow}
\usepackage{lmodern}
\usepackage{latexsym}
\usepackage{wasysym}

\newtheorem{theorem}{Theorem}

\newtheorem{corollary}{Corollary}

\newtheorem{example}{Example}
\numberwithin{equation}{section}

\begin{document}
\title[{Sheffer-Appell polynomial sequence: A matrix approach}]
{Differential equation and recurrence relations of the Sheffer-Appell polynomial sequence: A matrix approach}

\author[H. M. Srivastava, Saima Jabee and Mohammad Shadab]
{H. M. Srivastava, Saima Jabee and Mohammad Shadab$^{*}$}

\address{H. M. Srivastava: Department of Mathematics and Statistics.
University of Victoria,
Victoria, British Columbia V8W 3R4,
Canada; and Department of Medical Research,
China Medical University Hospital,
China Medical University,
Taichung 40402, Taiwan, Republic of China.}
\email{harimsri@math.uvic.ca}

\address{Saima Jabee: Department of Applied Sciences and Humanities,
Faculty of Engineering and Technology,
Jamia Millia Islamia (A Central University),
New Delhi 110025, India.}
\email{saimajabee007@gmail.com}

\address{Mohammad Shadab: Department of Applied Sciences and Humanities,
Faculty of Engineering and Technology,
Jamia Millia Islamia (A Central University),
New Delhi 110025, India.}
\email{shadabmohd786@gmail.com}



\thanks{*Corresponding author}

\begin{abstract}
Motivated by the effective impact
of the Pascal functional
and the Wronskian matrices,
we investigate several identities
and differential equation
for the Sheffer-Appell polynomial
sequence by using matrix algebra.
The matrix approach, which we
have used in this article, is convenient
to derive the generating functions of
the Sheffer-Appell polynomial sequence.
By means of examples, we apply and also
illustrate our results to
an extended class of polynomial sequences.
\end{abstract}

\maketitle

\section{introduction}
Sequences of polynomials play an important r\^{o}le
in many problems of pure and applied mathematical sciences such as those
occurring in approximation theory, statistics, combinatorics and analysis
(see, for example, \cite{Rosen, Roman1,Roman2,Roman3}).
The class of Sheffer sequences is one of the most important
classes of polynomial sequences. A polynomial sequence
$\{s_{n}(x)\}_{n=0}^{\infty}$ is called a Sheffer polynomial sequence
\cite{Bucchianico,Dattoli,Roman2,Sheffer2}
if and only if its generating function has the following form:
\begin{eqnarray}\label{eq(d1)}
A(y)e^{xH(y)}=\sum_{n=0}^{\infty}s_{n}(x)\;\frac{y^{n}}{n!},
\end{eqnarray}
where
\begin{eqnarray}
A(y)=A_{0}+A_{1}y+\cdots,\nonumber
\end{eqnarray}
and
\begin{eqnarray}
H(y)=H_{1}y+H_{2}y^2+\cdots,\nonumber
\end{eqnarray}
with $A_0 \neq 0$ and  $H_1 \neq 0$.

Let us recall an alternate definition of the Sheffer sequences
in terms of a pair of generating functions $(l(y),h(y))$
(see, for example, \cite{Lehmer}):\\

Let $h(y)$ be a delta series and let $l(y)$ be an invertible series, defined as follows:
\begin{eqnarray}\label{eq(d2)}
h(y)=\sum_{n=0}^{\infty}h_{n}\;\frac{y^{n}}{n!} \qquad (h(0)=0; \; h(1)\neq 0)
\end{eqnarray}
and
\begin{eqnarray}\label{eq(d3)}
l(y)=\sum_{n=0}^{\infty}l_{n}\;\frac{y^{n}}{n!} \qquad (l(0)\neq 0).
\end{eqnarray}
Then there exists a unique sequence of Sheffer polynomials $s_{n}(x)$
satisfying the orthogonality conditions:
\begin{eqnarray}\label{eq(d4)}
\langle l(y)h(y)^k|s_n(x)\rangle=n!\delta_{n,k} \qquad(\forall\; n,k \geqq 0),
\end{eqnarray}
where $\delta_{n,k}$ is the Kronecker delta.\\

Roman \cite[p. 18, Theorem 2.3.4]{Roman2} introduced the
exponential generating function of $s_n(x)$ as follows:
\begin{eqnarray}\label{eq(d5)}
\frac{1}{h^{-1}(y)}\;e^{xh^{-1}(y)}=\sum_{n=0}^{\infty}s_{n}(x)\; \frac{y^{n}}{n!}.
\end{eqnarray}
The Sheffer sequence for the pair $(l(y), y)$ is called an Appell sequence for $l(y)$.
In fact, Roman \cite{Roman2} characterized Appell sequences in several ways:\\

\{$A_{n}(x)\}_{n\in\mathbb{N}}$ is an Appell set if either
\begin{eqnarray}
\frac{d}{dx}\big(\alpha_n(x)\big)=n \alpha_n(x) \qquad{n\in\mathbb{N}}\nonumber
\end{eqnarray}
or if there exists an exponential generating function of the form
(see also the recent works \cite{Pinter,HMS1982}):
\begin{eqnarray}\label{eq(d6)}
A(y)e^{xy}=\sum_{n=0}^{\infty}\alpha_{n}(x)\;\frac{y^{n}}{n!},
\end{eqnarray}
where $\mathbb{N}$ denotes the set of positive integers and
\begin{eqnarray}
A(y)=\frac{1}{l(y)}.\nonumber
\end{eqnarray}
We also note that, for $H(y)=y$, the generating function \eqref{eq(d1)}
of the Sheffer polynomials $s_n(x)$ reduces to the generating function
\eqref{eq(d6)} of the Appell polynomials $\alpha_n(x)$.\\

The polynomials defined as the discrete convolution of known polynomials
are used to investigate new families of special functions.
For example, the polynomial $h_{n}^{(A)}(x)$ given by
\begin{eqnarray}\label{eq(d7)}
h_{n}^{(A)}(x)=\sum_{r=0}^{\infty}A_{k}h_{n-k}(x),
\end{eqnarray}
is known as a discrete Appell convolution by
setting $h_{n}(x)=x^n$ in the above equation.\\

In the year 2015, Subuhi {\it et al.} \cite{Subuhi} introduced
the determinantal definition and other properties of the Sheffer-Appell polynomials.
The Sheffer-Appell polynomial sequences are combination of the families of
the Sheffer and the Appell polynomials sequences.\\

Now, in order to recall the definition of the generalized Pascal functional
matrix of an analytic function (see \cite{Yang1}), let
$$\mathcal{F}=\left\{h(y)=\sum_{r=0}^{\infty}\alpha_r\;
\frac{y^r}{r!}\;\Bigg|\;\alpha_r\in\mathbb{C}\right\}$$
be the set of power series possessing the $\mathbb{C}$-algebra. Then
the generalized Pascal functional matrix
$[P_n(h(y))]$, which is a lower triangular matrix of order
$(n+1)\times(n+1)$ for $h(y)\in \mathcal{F}$,
is defined by
\begin{eqnarray}\label{eq(d8)}
P_n[h(y)]_{ij}=\left\{
\begin{array}{ll}
\dbinom{i}{j}h^{(i-j)}(y) & \qquad (i\geqq j)\\
\\
0 & \qquad  (\text{otherwise}),
\end{array}
\right.
\end{eqnarray}
for all $i,j=0,1,2,\cdots,n$. Here $h^{(i)}(y)$ is the $i${th}
order derivative of $h(y)$.\\

We next recall the $n${th} order Wronskian matrix of
several analytic functions $h_{1}(y),h_{2}(y),\cdots,h_{m}(y)$
of order $(n+1)\times m$ as follows:
\begin{eqnarray}\label{eq(d9)}
W_{n}[h_{1}(y),h_{2}(y),\cdots,h_{m}(y)]=
\begin{bmatrix}
h_{1}(y) & h_{2}(y) & h_{3}(y) & \cdots & h_{m}(y) \\[3pt]
h_{1}^{'}(y) & h_{2}^{'}(y) & h_{3}^{'}(y) & \cdots & h_{m}^{'}(y) \\[3pt]
\vdots & \vdots & \vdots & \ddots & \vdots \\[3pt]
h_{1}^{(n)}(y) & h_{2}^{(n)}(y) & h_{3}^{(n)}(y) & \cdots & h_{m}^{(n)}(y) \\
\end{bmatrix}.
\end{eqnarray}
We also record here some properties and relationships
between the Wronskian matrices and the generalized Pascal functional matrices
as they are the main tool of our work (see, for example, \cite{Yang2,Youn}).\\

\noindent
\textbf{Property I.} For $h(y),l(y)\in \mathbb{C}$,
$P_n[h(y)]$ and $W_n[h(y)]$ are linear,
that is,
\begin{eqnarray}
P_n[uh(y)+vl(y)]=uP_n[h(y)]+vP_n[l(y)]\nonumber
\end{eqnarray}
and
\begin{eqnarray}\label{eq(d10)}
W_n[uh(y)+vl(y)]=uW_n[h(y)]+vW_n[l(y)],
\end{eqnarray}
where $u,v\in\mathbb{C}$.\\

\noindent
\textbf{Property II.} For $h(y),l(y)\in \mathbb{C}$,
\begin{eqnarray}\label{eq(d11)}
P_n[h(y)l(y)]=P_n[h(y)]P_n[l(y)]=P_n[l(y)]P_n[h(y)].
\end{eqnarray}

\noindent
\textbf{Property III.} For $h(y),l(y)\in \mathbb{C}$,
\begin{eqnarray}\label{eq(d12)}
W_n[h(y)l(y)]=P_n[h(y)]W_n[l(y)]=P_n[l(y)]W_n[h(y)].
\end{eqnarray}

\noindent
\textbf{Property IV.} For $h(y),l(y)\in \mathbb{C}$,
with $h(0)=0$ and $h'(0)\neq 0$,
\begin{eqnarray}\label{eq(d13)}
W_n[l(h(y))]\big|_{y=0}=W_{n}\left[1,h(y),h^2(y),
\cdots,h^n(y)\right]\big|_{y=0}\; \Omega_{n}^{-1}W_n[l(y)]\big|_{y=0},
\end{eqnarray}
where $0!,1!,\cdots,n!$ are the diagonal entries in the
diagonal matrix given by
$$\Omega_{n}=\text{diag}[0!,1!,2!,\cdots,n!].$$

\section{The Sheffer-Appell polynomial sequence and its differential equation}

He and Ricci (\cite{He}; see also \cite{Sheffer1}) derived some
recurrence relations and differential equation for the Appell polynomial
sequence. Further, Youn and Yang (\cite{Youn}; see also \cite{Aceto})
obtained some identities and differential equation for the Sheffer polynomial
sequence by using matrix algebra.
Here, in this paper, we study some recursive formulas and
differential equation for the Sheffer-Appell polynomial sequence
by using matrix algebra.\\

The Sheffer-Appell polynomial sequence, which is denoted by ${_s}A_{n}(x)$,
is defined as the discrete
Appell convolution of the Sheffer polynomials $s_{n}(x)$.\\
The generating function of the Sheffer-Appell polynomials ${_s}A_{n}(x)$ is
given by
\begin{eqnarray}\label{eq(d14)}
\frac{1}{l(h^{-1}(y))l(y)}\;e^{{xh^{-1}(y)}}
=\sum_{n=0}^{\infty}{_s}A_{n}(x)\;\frac{y^{n}}{n!},
\end{eqnarray}
where $h^{-1}(y)$ is the compositional inverse of $h(y)$, that is, we have
(see \cite{Yang2})
$$h^{-1}(h(y))=(h(h^{-1}(y))=y.$$
Thus, if the following generating function in \eqref{eq(d14)}:
$$\frac{1}{l(h^{-1}(y))h(y)}\;e^{{xh^{-1}(y)}}$$
is analytic, then (by using Taylor's expansion theorem), we obtain
\begin{eqnarray}\label{eq(d15)}
{_s}A_{k}(x)=\frac{d^{k}}{dy^{k}}\left(\frac{1}{l(h^{-1}(y))l(y)}\;
e^{{xh^{-1}(y)}}\right)\Bigg{|}_{y=0}\qquad (k\geqq 0).
\end{eqnarray}
The Sheffer-Appell polynomial sequence ${_s}A_{n}(x)$ in vector form
for the pair $(l(y),h(y))$ is denoted by $\vec{{_s}A_{n}}(x)$ and it is defined by
\begin{eqnarray}\label{eq(d16)}
\vec{{_s}A_{n}}(x)=[{_s}A_{0}(x), {_s}A_{1}(x),\cdots,{_s}A_{n}(x)]^{T},
\end{eqnarray}
which can also be expressed as follows:
\begin{eqnarray}\label{eq(d17)}
\vec{{_s}A_{n}}(x)=[{_s}A_{0}(x), {_s}A_{1}(x),\cdots,{_s}A_{n}(x)]^{T}
=W_{n}\left[\frac{1}{l(h^{-1}(y))l(y)}e^{{xh^{-1}(y)}}\right]\Bigg{|}_{y=0}.
\end{eqnarray}

\noindent
{\bf Lemma.}
{\it Let ${_s}A_{n}(x)$ be the Sheffer-Appell polynomial sequence
for the pair $(l(y),h(y))$. Then}
\begin{align}\label{eq(d18)}
&W_{n}[{_s}A_{0}(x), {_s}A_{1}(x),\cdots,{_s}A_{n}(x)]^{T}
\Omega_{n}^{-1}\notag \\
&\qquad =W_{n}[1,(h^{-1}(y)),(h^{-1}(y))^{2},\cdots,
(h^{-1}(y))^{n}]\big{|}_{y=0}\nonumber\\
&\qquad \qquad \cdot \Omega_{n}^{-1}P_{n}\left[\frac{1}{l(y)}\right]\Bigg{|}_{y=0}\;
P_{n}\left[\frac{1}{l(h(y))}\right]\Bigg{|}_{y=0}P_{n}\left[e^{xy}\right]\Bigg{|}_{y=0}.
\end{align}

\begin{proof}
Let us begin with the equation \eqref{eq(d17)}, that is,
\begin{eqnarray}\label{eq(d19)}
\vec{{_s}A_{n}}(x)=W_{n}\left[\frac{1}{l(h^{-1}(y))l(y)}\;
e^{{xh^{-1}(y)}}\right]\Bigg{|}_{y=0}.
\end{eqnarray}
Applying Property IV in the equation \eqref{eq(d19)}, we get
\begin{eqnarray}\label{eq(d20)}
\vec{{_s}A_{n}}(x)=W_{n}[1,(h^{-1}(y)),(h^{-1}(y))^{2},
\cdots,(h^{-1}(y))^{n}]\Bigg{|}_{y=0}\Omega_{n}^{-1}W_{n}
\left[\frac{1}{l(y)l(h(y))}e^{{xy}}\right]\Bigg{|}_{y=0}.
\end{eqnarray}
In view of the following result:
\begin{eqnarray}\label{eq(d21)}
W_{n}[e^{xy}]\Bigg{|}_{y=0}=[1,\,\,\,x,\,\,\,x^{2},
\cdots,x^{n}]^{T},
\end{eqnarray}
the \eqref{eq(d20)} becomes
\begin{align}\label{eq(d22)}
\vec{{_s}A_{n}}(x)&=W_{n}[1,(h^{-1}(y)),(h^{-1}(y))^{2},\cdots,
(h^{-1}(y))^{n}]\Bigg{|}_{y=0}\nonumber\\
&\qquad \cdot\Omega_{n}^{-1}P_{n}\left[\frac{1}{l(y)}\right]
\Bigg{|}_{y=0}P_{n}\left[\frac{1}{l(h(y))}\right]
\Bigg{|}_{y=0}[1,\,\,\, x,\,\,\, x^{2},\cdots, x^{n}]^{T}.
\end{align}

Now, by taking the $k${th} order derivative of both sides of
the equation \eqref{eq(d22)} with respect to $x$
and dividing the resulting equation by $k!$, we obtain
\begin{align}\label{eq(d23)}
&\frac{1}{k!}[{_s}A_{0}^{(k)}(x), {_s}A_{1}^{(k)}(x),
\cdots,{_s}A_{n}^{(k)}(x)]^{T}\notag \\
&\qquad =W_{n}[1,(h^{-1}(y)),(h^{-1}(y))^{2},\cdots,
(h^{-1}(y))^{n}]\Bigg{|}_{y=0}\nonumber\\
&\qquad \qquad \cdot\Omega_{n}^{-1}P_{n}\left[\frac{1}{l(y)}\right]
\Bigg{|}_{y=0}P_{n}\left[\frac{1}{l(h(y))}\right]
\Bigg{|}_{y=0}\nonumber\\
&\qquad \qquad \cdot\left[0,\,\,\,\cdots\,\,\,,0,\,\,\,1,\,\,\,
\binom{k+1}{k}x,\,\,\,\binom{k+1}{k}x^2,\cdots,
\binom{n}{k} x^{n-k}\right]^T.
\end{align}
Hence, clearly, the right-hand side and left-hand side of
the equation \eqref{eq(d23)} are the $k${th} columns of
$$W_{n}[1,(h^{-1}(y)),(h^{-1}(y))^{2},\cdots,
(h^{-1}(y))^{n}]\Bigg{|}_{y=0}\Omega_{n}^{-1},$$
$$P_{n}\left[\frac{1}{l(y)}\right]\Bigg{|}_{y=0}\;P_{n}
\left[\frac{1}{l(h(y))}\right]\Bigg{|}_{y=0}P_{n}[e^{xy}]
\Bigg{|}_{y=0}$$
and
$$W_{n}[{_s}A_{0}(x), {_s}A_{1}(x),\cdots,{_s}A_{n}(x)]^T \Omega^{-1},$$
respectively. Our proof of the Lemma is thus completed.
\end{proof}

We now state and prove Theorem \ref{thm2.1} below.

\begin{theorem}\label{thm2.1}
The Sheffer-Appell polynomial sequence ${_s}A_{n}(x)\thicksim(l(y),h(y))$
satisfies the following differential equation$:$
\begin{eqnarray}\label{eq(d24)}
\sum_{k=0}^{n}(xa_k+b_k+c_k)\frac{{_s}A_{n}^{(k)}(x)}{k!}-n{_s}A_{n}(x)=0,
\end{eqnarray}
where
$$a_k=\left(\frac{h(y)}{h'(y)}\right)^{(k)}\Bigg|_{y=0}\qquad (k\geqq 0),$$
$$b_k=\left(-\frac{h(y)l'(h(y))}{l(h(y))}\right)^{(k)}\Bigg|_{y=0}
\qquad (k\geqq 0)$$
and
$$c_k=\left(-\frac{h(y)l'(y)}{h'(y)l(y)}\right)^{(k)}\Bigg|_{y=0}
\qquad (k\geqq 0).$$
\end{theorem}

\begin{proof}
Let us begin with the following result:
\begin{eqnarray}\label{eq(d25)}
W_{n}\left[y \frac{d}{dy}\left(\frac{e^{xh^{-1}(y)}}{l(h^{-1}(y))l(y)}
\right)\right]\Bigg|_{y=0}.
\end{eqnarray}
On the one hand, by using Property III, we get
\begin{align}\label{eq(d26)}
&W_{n}\left[ y \frac{d}{dy}\left(\frac{e^{xh^{-1}(y)}}
{l(h^{-1}(y))l(y)}\right)\right]\Bigg|_{y=0}\notag \\
&\qquad =P_n[y]\Bigg|_{y=0}\; W_{n}\left[\frac{d}{dy}
\left(\frac{e^{xh^{-1}(y)}}{l(h^{-1}(y))l(y)}\right)\right]\Bigg|_{y=0}
\notag \\
&\qquad =
\begin{bmatrix}
0 & 0 & 0 & \cdots & 0 & 0 & 0 \\
1 & 0 & 0 & \cdots & 0 & 0 & 0 \\
0 & 2 & 0 & \cdots & 0 & 0 & 0 \\
0 & 0 & 3 & \cdots & 0 & 0 & 0 \\
\vdots & \vdots & \vdots & \ddots & \vdots & \vdots & \vdots \\
0 & 0 & 0 & \cdots & 0 & 0 & 0 \\
0 & 0 & 0 & \cdots & n-1 & 0 & 0 \\
0 & 0 & 0 & \cdots & 0 & n & 0 \\
\end{bmatrix}
\begin{bmatrix}
{_s}A_{1}(x)  \\
{_s}A_{2}(x)  \\
{_s}A_{3}(x)  \\
\vdots  \\
{_s}A_{n}(x)  \\
{_s}A_{n+1}(x)  \\
\end{bmatrix}.
\end{align}
Also, on the other hand, we can rewrite the equation \eqref{eq(d25)} as
follows:
\begin{align}\label{eq(d27)}
&W_{n}\left[y \frac{d}{dy}\left(\frac{e^{xh^{-1}(y)}}{l(h^{-1}(y))l(y)}
\right)\right]\Bigg|_{y=0}\nonumber\\
&\qquad=W_{n}\left[ \left(x\frac{h(h^{-1}(y))}{h'(h^{-1}(y))}
-\frac{h(h^{-1}(y))l'(y)}{l(y)}-\frac{h(h^{-1}(y))}{h'(h^{-1}(y))}
\;\frac{l'(h^{-1}(y))}{l(h^{-1}(y))}\right)\right.\notag \\
&\qquad \qquad \left. \cdot \frac{e^{xh^{-1}(y)}}{l(h^{-1}(y))l(y)}
\right]\Bigg|_{y=0}.
\end{align}
Thus, by using Property IV in the equation \eqref{eq(d27)}, we have
\begin{align}\label{eq(d28)}
&W_{n}\left[ y \frac{d}{dy}\left(\frac{e^{xh^{-1}(y)}}{l(h^{-1}(y))l(y)}
\right)\right]\Bigg|_{y=0}\notag \\
&\qquad =W_{n}[1,(h^{-1}(y)),(h^{-1}(y))^{2},\cdots,(h^{-1}(y))^{n}]
\Bigg{|}_{y=0}\Omega_{n}^{-1}\nonumber\\
&\qquad \qquad \cdot W_{n}\left[\left(x\;\frac{h(y)}{h'(y)}
-\frac{h(y)l'(h(y))}{l(h(y))}
-\frac{h(y)}{h'(y)}\frac{l'(y)}{l(y)}\right)\frac{e^{xy}}
{l(y)l(h(y))}\right]\Bigg|_{y=0}.
\end{align}

Next, by using Property III in the equation \eqref{eq(d28)}, we get
\begin{align}\label{eq(d29)}
&W_{n}\left[ y \frac{d}{dy}\left(\frac{e^{xh^{-1}(y)}}
{l(h^{-1}(y))l(y)}\right)\right]\Bigg|_{y=0}\nonumber\\
&\qquad=W_{n}[1,(h^{-1}(y)),(h^{-1}(y))^{2},\cdots,(h^{-1}(y))^{n}]
\Bigg{|}_{y=0}\Omega_{n}^{-1} P_{n}\left[e^{xy}\right]\Bigg{|}_{y=0}\;
P_{n}\left[\frac{1}{l(y)}\right]\Bigg{|}_{y=0}\nonumber\\
&\qquad \qquad \cdot P_n\left[\frac{1}{l(h(y))}\right]\Bigg{|}_{y=0}\;
W_{n}\left[ \left(x\frac{h(y)}{h'(y)}-\frac{h(y)l'(h(y))}{l(h(y))}
-\frac{h(y)}{h'(y)}\frac{l'(y)}{l(y)}\right)\right]\Bigg|_{y=0},
\end{align}
which, by applying the above Lemma, yields
\begin{align*}
&W_{n}\left[ y \frac{d}{dy}\left(\frac{e^{xh^{-1}(y)}}
{l(h^{-1}(y))l(y)}\right)\right]\Bigg|_{y=0}
=W_{n}[{_s}A_{0}(x), {_s}A_{1}(x),\cdots,{_s}A_{n}(x)]^T\notag \\
&\qquad \cdot \Omega^{-1}W_{n}\left[\left(x\;\frac{h(y)}
{h'(y)}-\frac{h(y)l'(h(y))}{l(h(y))}
-\frac{h(y)}{h'(y)}\frac{l'(y)}{l(y)}\right)\right]\Bigg|_{y=0}
\end{align*}
or, equivalently,
\begin{align}\label{eq(d30)}
&W_{n}\left[y \frac{d}{dy}\left(\frac{e^{xh^{-1}(y)}}
{l(h^{-1}(y))l(y)}\right)\right]\Bigg|_{y=0}\nonumber\\
&\qquad =
\begin{bmatrix}
{_s}A_{0}(x) & 0 & 0 & \cdots & 0 \\[3pt]
{_s}A_{1}(x) & \frac{{_s}A_{1}^{'}(x)}{1!} & 0 & \cdots & 0 \\[3pt]
{_s}A_{2}(x) & \frac{{_s}A_{2}^{'}(x)}{1!}
& \frac{{_s}A_{2}^{''}(x)}{2!} & \cdots & 0 \\[3pt]
\vdots & \vdots & \vdots & \ddots & \vdots \\[3pt]
{_s}A_{n}(x) & \frac{{_s}A_{n}^{'}(x)}{1!}
& \frac{{_s}A_{n}^{''}(x)}{2!} & \cdots & \frac{{_s}A_{n}^{(n)}(x)}{n!} \\
\end{bmatrix}
\begin{bmatrix}
xa_0+b_0+c_0  \\
xa_1+b_1+c_1  \\
xa_2+b_2+c_2  \\
\vdots  \\
xa_n+b_n+c_n  \\
\end{bmatrix}.
\end{align}

Finally, upon equating the $n${th} rows of \eqref{eq(d26)} and \eqref{eq(d30)},
we obtain the desired result \eqref{eq(d24)} asserted by Theorem \ref{thm2.1}.
\end{proof}

By setting $l(y)=1$ and $b_{k}=c_{k}=0$ $(\forall\; k\geqq 1)$ in Theorem \ref{thm2.1},
we get the following corollary.

\begin{corollary} \label{cor2.1}
Let ${_q}A_{n}(x)\thicksim(1,h(y))$ be the associated polynomial sequence.
Then
\begin{eqnarray}\label{eq(d31)}
x\sum_{k=0}^{n}a_k\frac{{_q}A_{n}^{(k)}(x)}{k!}-n{_q}A_{n}(x)=0.
\end{eqnarray}
\end{corollary}

In its special case when  $a_k=(1)_{k}$,
$b_{k}=(\lambda+1)((2)_{k}-(1)_{k})$ and $c_{k}=(\lambda+1)[(3)_{k}-(4)_{k}]$,
Theorem \ref{thm2.1} would apply to the Laguerre polynomials as follows.

\begin{corollary} \label{cor2.2}
Let
$$L^{(\lambda)}_{n}(x)\thicksim\left((1-y)^{-\lambda-1},\frac{y}{y-1}\right)$$
be the generalized Laguerre polynomial of degree
$n$ in $x$ and with the index (or order) $\lambda$.
Then
\begin{eqnarray}\label{eq(d32)}
\sum_{k=1}^{n}\binom{n}{k}k!\left(x-\frac{k(k-1)(k+4)
(\lambda+1)}{6}\right){_L}A_{n-k}(x)=n\,{_L}A_{n}(x).
\end{eqnarray}
\end{corollary}

\begin{example} \label{ex2.1}
{\rm By applying Theorem \ref{thm2.1} to the Miller-Lee type Appell polynomials
$G_{n}^{(m)}(x)$ given by
$$G_{n}^{(m)}(x)\thicksim\left((1-y)^{m+1}, y\right),$$
we have
\begin{eqnarray*}
a_{k}=\left\{
\begin{array}{ll}
1 & \qquad (k=1)\\
\\
0 & \qquad (k>1),
\end{array}
\right.
\end{eqnarray*}
\begin{eqnarray*}
b_{k}=\left\{
\begin{array}{ll}
0 &\qquad  (k=0)\\
\\
(m+1)(1)_k & \qquad (k>0)
\end{array}
\right.
\end{eqnarray*}
and
\begin{eqnarray*}
c_{k}=\left\{
\begin{array}{ll}
0 & \qquad (k=0)\\
\\
-(m+1)(1)_k & \qquad (k>0).
\end{array}
\right.
\end{eqnarray*}
Hence we get the following recurrence relation for the Miller-Lee type
Appell polynomials $G_{n}^{(m)}(x):$}
\begin{eqnarray}\label{eq(d49a)}
n{_G}A_{n}(x)-nx{_G}A_{n-1}(x)=\sum_{k=1}^{n}\binom{n}{k}{_G}A_{n-k}(x)(b_{k}+c_{k}).
\end{eqnarray}
\end{example}

\section{Recurrence relations for the Sheffer-Appell polynomials}

Here, in this section, we first state and prove Theorem \ref{thm3.1} below.

\begin{theorem} \label{thm3.1}
Let ${_s}A_{n}(x) \thicksim(l(y),h(y))$ be the Sheffer-Appell polynomial sequence.
Then the following recursive formula holds true for ${_s}A_{n}(x):$
\begin{eqnarray}\label{eq(d33)}
{_s}A_{n+1}(x)=\sum_{k=0}^{n}(xa_k+b_k+c_k)\;\frac{{_s}A_{n}^{(k)}(x)}{k!},
\end{eqnarray}
where
$$a_k=\left(\frac{1}{h'(y)}\right)^{(k)}\Bigg|_{y=0}\qquad (k \geqq 0),$$
$$b_k=\left(-\frac{l'(h(y))}{l(h(y))}\right)^{(k)}\Bigg|_{y=0}\qquad (k \geqq 0)$$
and
$$c_k=\left(-\frac{l'(y)}{h'(y)l(y)}\right)^{(k)}\Bigg|_{y=0}\qquad (k \geqq 0).$$
\end{theorem}

\begin{proof}
Let us consider
\begin{eqnarray}\label{eq(d34)}
W_{n}\left[\frac{d}{dy}\left(\frac{e^{xh^{-1}(y)}}
{l(h^{-1}(y))l(y)}\right)\right]\Bigg|_{y=0},
\end{eqnarray}
which, on the one hand, can be written as follows:
\begin{eqnarray}\label{eq(d35)}
W_{n}\left[\frac{d}{dy}\left(\frac{e^{xh^{-1}(y)}}
{l(h^{-1}(y))l(y)}\right)\right]\Bigg|_{y=0}
=[{_s}A_{1}(x), {_s}A_{2}(x),\cdots,{_s}A_{n+1}(x)]^T.
\end{eqnarray}
Also, on the other hand, we can write the equation \eqref{eq(d34)} in
the following form:
\begin{eqnarray}\label{eq(d36)}
W_{n}\left[\left(x\frac{1}{h'(h^{-1}(y))}-\frac{l'(y)}{l(y)}
-\frac{1}{h'(h^{-1}(y))}\frac{l'(h^{-1}(y))}{l(h^{-1}(y))}\right)
\frac{e^{xh^{-1}(y)}}{l(h^{-1}(y))l(y)}\right]\Bigg|_{y=0},
\end{eqnarray}
which, on using Property IV, yields
\begin{align}\label{eq(d37)}
&W_{n}\left[\left(x\frac{1}{h'(h^{-1}(y))}-\frac{l'(y)}{l(y)}
-\frac{1}{h'(h^{-1}(y))}\frac{l'(h^{-1}(y))}{l(h^{-1}(y))}\right)
\frac{e^{xh^{-1}(y)}}{l(h^{-1}(y))l(y)}\right]\Bigg|_{y=0}\notag \\
&\qquad =W_{n}[1,(h^{-1}(y)),(h^{-1}(y))^{2},\cdots,(h^{-1}(y))^{n}]
\Bigg{|}_{y=0}\; \Omega_{n}^{-1}\nonumber\\
&\qquad \qquad \cdot W_{n}\left[ \left(x\frac{1}{h'(y)}
-\frac{l'(h(y))}{l(h(y))}-\frac{1}{h'(y)}\frac{l'(y)}{l(y)}\right)
\frac{e^{xy}}{l(y)l(h(y))}\right]\Bigg{|}_{y=0}.
\end{align}

Now, if we make use of Property III, we find from \eqref{eq(d37)}
that
\begin{align}\label{eq(d38)}
&W_{n}\left[\left(x\frac{1}{h'(h^{-1}(y))}-\frac{l'(y)}{l(y)}
-\frac{1}{h'(h^{-1}(y))}\frac{l'(h^{-1}(y))}{l(h^{-1}(y))}\right)
\frac{e^{xh^{-1}(y)}}{l(h^{-1}(y))l(y)}\right]\Bigg|_{y=0}\notag \\
&\qquad =W_{n}[1,(h^{-1}(y)),(h^{-1}(y))^{2},\cdots,(h^{-1}(y))^{n}]
\Bigg{|}_{y=0}\;
\Omega_{n}^{-1}\nonumber\\
&\qquad \qquad \cdot P_{n}\left[e^{xy}\right]\Bigg{|}_{y=0}\;
P_{n}\left[\frac{1}{l(y)}\right]\Bigg{|}_{y=0}\;
P_n\left[\frac{1}{l(h(y))}\right]\Bigg{|}_{y=0}\notag \\
&\qquad \qquad \cdot W_{n}\left[\left(x\frac{1}{h'(y)}
-\frac{l'(h(y))}{l(h(y))}-\frac{1}{h'(y)}\frac{l'(y)}{l(y)}\right)
\right]\Bigg{|}_{y=0}.
\end{align}

Finally, by applying the above Lemma, we get
\begin{align*}
&W_{n}\left[\left(x\frac{1}{h'(h^{-1}(y))}-\frac{l'(y)}{l(y)}
-\frac{1}{h'(h^{-1}(y))}\frac{l'(h^{-1}(y))}{l(h^{-1}(y))}\right)
\frac{e^{xh^{-1}(y)}}{l(h^{-1}(y))l(y)}\right]\Bigg|_{y=0}\notag \\
&\qquad =W_{n}[{_s}A_{0}(x), {_s}A_{1}(x),\cdots,{_s}A_{n}(x)]^T \Omega^{-1}
\notag \\
&\qquad \qquad \cdot W_{n}\left[ \left(x\frac{1}{h'(y)}-\frac{l'(h(y))}{l(h(y))}
-\frac{1}{h'(y)}\frac{l'(y)}{l(y)}\right)\right]\Bigg{|}_{y=0}
\end{align*}
or, equivalently,
\begin{align}\label{eq(d39)}
&W_{n}\left[\left(x\frac{1}{h'(h^{-1}(y))}-\frac{l'(y)}{l(y)}
-\frac{1}{h'(h^{-1}(y))}\frac{l'(h^{-1}(y))}{l(h^{-1}(y))}\right)
\frac{e^{xh^{-1}(y)}}{l(h^{-1}(y))l(y)}\right]\Bigg|_{y=0}\notag \\
&\qquad =
\begin{bmatrix}
{_s}A_{0}(x) & 0 & 0 & \cdots & 0 \\[3pt]
{_s}A_{1}(x) & \frac{{_s}A_{1}^{'}(x)}{1!} & 0 & \cdots & 0 \\[3pt]
{_s}A_{2}(x) & \frac{{_s}A_{2}^{'}(x)}{1!} & \frac{{_s}A_{2}^{''}(x)}{2!}
& \cdots & 0 \\[3pt]
\vdots & \vdots & \vdots & \ddots & \vdots \\[3pt]
{_s}A_{n}(x) & \frac{{_s}A_{n}^{'}(x)}{1!} & \frac{{_s}A_{n}^{''}(x)}{2!}
& \cdots & \frac{{_s}A_{n}^{(n)}(x)}{n!} \\
\end{bmatrix}
\begin{bmatrix}
xa_0+b_0+c_0  \\
xa_1+b_1+c_1  \\
xa_2+b_2+c_2  \\
\vdots  \\
xa_n+b_n+c_n  \\
\end{bmatrix}.
\end{align}
Equating the $n${th} rows of \eqref{eq(d35)} and \eqref{eq(d39)},
we arrive at the desired result \eqref{eq(d33)} asserted by Theorem \ref{thm3.1}.
\end{proof}

Corollary \ref{cor3.1} below follows from Theorem \ref{thm3.1} in its special case when
$l(y)=1$ and $b_{k}=c_{k}=0$ $(\forall\; k \geqq 1)$ in the recursive formula \eqref{eq(d33)}.

\begin{corollary}\label{cor3.1}
Let ${_q}A_{n}(x)\thicksim(1,h(y))$ be the associated polynomial
sequence. Then
\begin{eqnarray}\label{eq(d40)}
{_q}A_{n+1}(x)=x\sum_{k=0}^{n}a_k\;\frac{{_q}A_{n}^{(k)}(x)}{k!},
\end{eqnarray}
where
$$a_{k}=\left(\frac{1}{h'(y)}\right)^{(k)}\Bigg{|}_{y=0}\qquad (k \geqq 0).$$
\end{corollary}

\begin{example} \label{ex3.1}
{\rm Here, in this example, we apply Theorem \ref{thm3.1} to the
generalized Laguerre polynomials $L_{n}^{(\lambda)}(x)$ given by
$$L_{n}^{\lambda}(x)\thicksim\left((1-y)^{-\lambda-1},
\frac{y}{y-1}\right)$$
with
\begin{eqnarray*}
a_{k}=\left\{
\begin{array}{ll}
-1 & \qquad (k=0)\\
\\
2 & \qquad (k=1)\\
\\
-2 &\qquad (k=2)\\
\\
0 & \qquad (k>2),
\end{array}
\right.
\end{eqnarray*}
\begin{eqnarray*}
b_{k}=-(\lambda+1)(1)_{k}
\end{eqnarray*}
and
\begin{eqnarray*}
c_{k}=\left\{
\begin{array}{ll}
-\lambda-1 & \qquad (k=0)\\
\\
\lambda+1 & \qquad (k=1)\\
\\
0 & \qquad (k>1).
\end{array}
\right.
\end{eqnarray*}
Hence we get the following recurrence relation
for the Laguerre-Appell polynomials:
\begin{align}\label{eq(d41)}
{_L}A_{n+1}(x)+(x+2\lambda+2){_L}A_{n}(x)
&=2xn{_L}A_{n-1}(x)-2(x+\lambda+1)\binom{n}{2}{_L}A_{n-2}(x)
\nonumber\\
&\qquad +(\lambda+1)\sum_{k=3}^{n}\binom{n}{k}{_L}A_{n-k}(x)k!.
\end{align}

If we apply Theorem \ref{thm3.1} to the Miller-Lee type Appell
polynomials $G_{n}^{(m)}(x)$ given by
$$G_{n}^{(m)}(x)\thicksim\left((1-y)^{m+1},y\right),$$
we have
\begin{eqnarray*}
a_{k}=\left\{
\begin{array}{ll}
1 & \qquad (k=0)\\
\\
0 & \qquad (k>0),
\end{array}
\right.
\end{eqnarray*}
\begin{eqnarray*}
b_{k}=(m+1)(1)_k
\end{eqnarray*}
and
\begin{eqnarray*}
c_{k}=-(m+1)(1)_k.
\end{eqnarray*}
Hence we get the following recurrence relation
for the Miller-Lee type Appell polynomials:}
\begin{eqnarray}\label{eq(d49b)}
{_G}A_{n+1}(x)-x{_G}A_{n}(x)=\sum_{k=0}^{n}
\binom{n}{k}{_G}A_{n-k}(x)(b_{k}+c_{k}).
\end{eqnarray}
\end{example}

\begin{theorem}\label{thm3.2}
Let $\{{_s}A_{n}(x)\} \thicksim(l(y),h(y))$ be the
Sheffer-Appell polynomial sequence.
Then the following recursive formula holds true for ${_s}A_{n}(x):$
\begin{eqnarray}\label{eq(d42)}
{_s}A_{n+1}(x)a_{0}=x {_s}A_{n}(x)
+\sum_{k=0}^{n}\binom{n}{k}{_s}A_{n-k}(x)(b_{k}+c_{k})
-\sum_{k=1}^{n}\binom{n}{k}{_s}A_{n+1-k}(x)a_{k},
\end{eqnarray}
where
$$a_k=\left({h'(h^{-1}(y))}\right)^{(k)}\Bigg{|}_{y=0},$$
$$b_k=\left(-\frac{h'(h^{-1}(y))l'(y)}{l(y)}\right)^{(k)}\Bigg{|}_{y=0}$$
and
$$c_k=\left(-\frac{l'(h^{-1}(y))}{l(h^{-1}(y))}\right)^{(k)}\Bigg{|}_{y=0}.$$
\end{theorem}

\begin{proof}
Let us begin with
\begin{eqnarray}\label{eq(d43)}
W_{n}\left[h'(h^{-1}(y))\frac{d}{dy}\left(\frac{e^{xh^{-1}(y)}}
{l(h^{-1}(y))l(y)}\right)\right]\Bigg{|}_{y=0},
\end{eqnarray}
which, by applying Property III, yields
\begin{align*}
&W_{n}\left[h'(h^{-1}(y))\frac{d}{dy}\left(\frac{e^{xh^{-1}(y)}}
{l(h^{-1}(y))l(y)}\right)\right]\Bigg{|}_{y=0}\notag \\
&\qquad =P_{n}\left[\frac{d}{dy}\left(\frac{e^{xh^{-1}(y)}}
{l(h^{-1}(y))l(y)}\right)\right]\Bigg{|}_{y=0}W_{n}
\left[h'(h^{-1}(y))\right]\Bigg{|}_{y=0}
\end{align*}
or, equivalently,
\begin{align}\label{eq(d44)}
&W_{n}\left[h'(h^{-1}(y))\frac{d}{dy}\left(\frac{e^{xh^{-1}(y)}}
{l(h^{-1}(y))l(y)}\right)\right]\Bigg{|}_{y=0}\notag \\
&\qquad =
\begin{bmatrix}
{_s}A_{1}(x) & 0 & 0 & \cdots & 0 \\[3pt]
{_s}A_{2}(x) & {_s}A_{1}(x) & 0 & \cdots & 0 \\[3pt]
{_s}A_{3}(x) & \binom{2}{1}{_s}A_{2}(x) & {_s}A_{1}(x) & \cdots & 0 \\[3pt]
\vdots & \vdots & \vdots & \ddots & \vdots \\[3pt]
{_s}A_{n+1}(x) & \binom{n}{1}{_s}A_{n}(x) & \binom{n}{2}{_s}A_{n-1}(x)
& \cdots & {_s}A_{1}(x) \\
\end{bmatrix}
\begin{bmatrix}
a_0  \\
a_1  \\
a_2  \\
\vdots  \\
a_n  \\
\end{bmatrix}.
\end{align}
On the other hand, we can write the equation \eqref{eq(d43)} as follows:
\begin{align}\label{eq(d45)}
&W_{n}\left[h'(h^{-1}(y))\frac{d}{dy}\left(\frac{e^{xh^{-1}(y)}}
{l(h^{-1}(y))l(y)}\right)\right]\Bigg{|}_{y=0}\notag \\
&\qquad = W_{n}\left[\left(x-\frac{h'(h^{-1}(y))l'(y)}{l(y)}
-\frac{l'(h^{-1}(y))}{l(h^{-1}(y))}\right)\frac{e^{xh^{-1}(y)}}
{l(h^{-1}(y))l(y)}\right]\Bigg{|}_{y=0}.
\end{align}

Now, if we apply Property I, we get
\begin{align}\label{eq(d46)}
&W_{n}\left[h'(h^{-1}(y))\frac{d}{dy}\left(\frac{e^{xh^{-1}(y)}}
{l(h^{-1}(y))l(y)}\right)\right]\Bigg{|}_{y=0}\notag \\
&\qquad =x W_{n} \left[\frac{e^{xh^{-1}(y)}}{l(h^{-1}(y))l(y)}
\right]\Bigg{|}_{y=0}
-W_{n}\left[\frac{h'(h^{-1}(y))l'(y)}{l(y)}\frac{e^{xh^{-1}(y)}}
{l(h^{-1}(y))l(y)}\right]\Bigg{|}_{y=0}\nonumber\\
&\qquad -W_{n}\left[\frac{l'(h^{-1}(y))}{l(h^{-1}(y))}\;
\frac{e^{xh^{-1}(y)}}
{l(h^{-1}(y))l(y)}\right]\Bigg{|}_{y=0},
\end{align}
which, by using Property III, yields
\begin{align*}
&W_{n}\left[h'(h^{-1}(y))\frac{d}{dy}\left(\frac{e^{xh^{-1}(y)}}
{l(h^{-1}(y))l(y)}\right)\right]\Bigg{|}_{y=0}\notag \\
&\qquad =x W_{n} \left[\frac{e^{xh^{-1}(y)}}
{l(h^{-1}(y))l(y)}\right]
\Bigg{|}_{y=0}+P_{n}\left[\frac{e^{xh^{-1}(y)}}
{l(h^{-1}(y))l(y)}\right]
\Bigg{|}_{y=0}W_{n}\left[\frac{-h'(h^{-1}(y))l'(y)}{l(y)}
\right]\Bigg{|}_{y=0}
\nonumber\\
&\qquad +P_{n}\left[\frac{e^{xh^{-1}(y)}}{l(h^{-1}(y))l(y)}
\right]\Bigg{|}_{y=0}\;
W_{n}\left[\frac{-l'(h^{-1}(y))}{l(h^{-1}(y))}\right]\Bigg{|}_{y=0}
\end{align*}
or, equivalently,
\begin{align}\label{eq(d47)}
&W_{n}\left[h'(h^{-1}(y))\frac{d}{dy}\left(\frac{e^{xh^{-1}(y)}}
{l(h^{-1}(y))l(y)}\right)\right]\Bigg{|}_{y=0}\notag \\
&\qquad =x
\begin{bmatrix}
{_s}A_{0}(x)  \\
{_s}A_{1}(x)  \\
{_s}A_{2}(x)  \\
\vdots  \\
{_s}A_{n-1}(x)  \\
{_s}A_{n}(x)  \\
\end{bmatrix}
+
\begin{bmatrix}
{_s}A_{0}(x) & 0 & 0 & \cdots & 0 \\[3pt]
{_s}A_{1}(x) & {_s}A_{0}(x) & 0 & \cdots & 0 \\[3pt]
{_s}A_{2}(x) & \binom{2}{1}{_s}A_{1}(x) & {_s}A_{0}(x) & \cdots & 0 \\[3pt]
\vdots & \vdots & \vdots & \ddots & \vdots \\[3pt]
{_s}A_{n}(x) & \binom{n}{1}{_s}A_{n-1}(x) & \cdots & {_s}A_{0}(x) \\
\end{bmatrix}
\begin{bmatrix}
b_0  \\
b_1  \\
b_2  \\
\vdots  \\
b_n  \\
\end{bmatrix}
\nonumber\\
&\qquad \qquad +
\begin{bmatrix}
{_s}A_{0}(x) & 0 & 0 & \cdots & 0 \\[3pt]
{_s}A_{1}(x) & {_s}A_{0}(x) & 0 & \cdots & 0 \\[3pt]
{_s}A_{2}(x) & \binom{2}{1}{_s}A_{1}(x) & {_s}A_{0}(x) & \cdots & 0 \\[3pt]
\vdots & \vdots & \vdots & \ddots & \vdots \\[3pt]
{_s}A_{n}(x) & \binom{n}{1}{_s}A_{n-1}(x) & \cdots & {_s}A_{0}(x) \\
\end{bmatrix}
\begin{bmatrix}
c_0  \\
c_1  \\
c_2  \\
\vdots  \\
c_n  \\
\end{bmatrix}.
\end{align}
Equating $n${th} rows of \eqref{eq(d44)} and \eqref{eq(d47)}, we arrive at
the desired result \eqref{eq(d42)} asserted by Theorem \ref{thm3.2}.
\end{proof}

Upon setting $l(y)=1$ and $b_{k}=c_{k}=0$ $(\forall \;k\geqq 1)$ in
the recursive formula \eqref{eq(d42)} asserted by Theorem \ref{thm3.1},
we can deduce the following corollary.

\begin{corollary}\label{cor3.2}
Let ${_q}A_{n}(x)\thicksim(1,h(y))$ be the associated polynomial sequence.
Then
\begin{eqnarray}\label{eq(d48)}
{_q}A_{n+1}(x)=x\sum_{k=0}^{n}a_k\; \frac{{_q}A_{n}^{(k)}(x)}{k!},
\end{eqnarray}
where
$$a_{k}=\left(\frac{1}{h'(y)}\right)^{(k)}\Bigg{|}_{y=0}.$$
\end{corollary}

\begin{example}\label{ex3.2}
{\rm Applying Theorem \ref{thm3.2} to the Miller-Lee type Appell polynomials
$G_{n}^{(m)}(x)$ given by
$$G_{n}^{(m)}(x)\thicksim\left((1-y)^{m+1}, y\right),$$ we have
\begin{eqnarray*}
a_{k}=\left\{
\begin{array}{ll}
1 & \qquad (k=0)\\
\\
0 & \qquad (k>0)
\end{array}
\right.
\end{eqnarray*}
and
\begin{eqnarray*}
b_{k}=c_k=-(\lambda+1)(1)_{k}.
\end{eqnarray*}
Hence we get the following recurrence relation for Miller-Lee type Appell polynomials:}
\begin{eqnarray}\label{eq(d49c)}
{_G}A_{n+1}(x)=x{_G}A_{n}(x)-2(\lambda+1)\sum_{k=0}^{n}\binom{n}{k}{_G}A_{n-k}(x)k!.
\end{eqnarray}
\end{example}

\begin{theorem}\label{thm3.3}
Let ${_s}A_{n}(x) \thicksim(l(y),h(y))$ be the Sheffer-Appell polynomial sequence.
Then the following recursive formula holds true for ${_s}A_{n}(x):$
\begin{eqnarray}\label{eq(d50)}
{_s}A_{n+1}(x)=\sum_{k=0}^{n}\binom{n}{k}(xa_k+b_k+c_k){_s}A_{n-k}(x),
\end{eqnarray}
where
$$a_k=\left(\frac{1}{h'(h^{-1}(y))}\right)^{(k)}\Bigg{|}_{y=0},$$
$$b_k=\left(-\frac{l'(y)}{l(y)}\right)^{(k)}\Bigg{|}_{y=0}$$
and
$$c_k=\left(-\frac{l'(h^{-1}(y))}{h'(h^{-1}(y))l(h^{-1}(y))}\right)^{(k)}
\Bigg{|}_{y=0}.$$
\end{theorem}

\begin{proof}
Our demonstration of Theorem \ref{thm3.3} begins with
\begin{eqnarray}\label{eq(d51)}
W_{n}\left[\frac{d}{dy}\left(\frac{e^{xh^{-1}(y)}}
{l(h^{-1}(y))l(y)}\right)\right]\Bigg{|}_{y=0},
\end{eqnarray}
which, on the one hand, can be rewritten as follows:
\begin{eqnarray}\label{eq(d52)}
W_{n}\left[\frac{d}{dy}\left(\frac{e^{xh^{-1}(y)}}
{l(h^{-1}(y))l(y)}\right)\right]\Bigg{|}_{y=0}
=[{_s}A_{1}(x), {_s}A_{2}(x),\cdots,{_s}A_{n+1}(x)]^T.
\end{eqnarray}
On the other hand, we can write \eqref{eq(d51)} in the following form:
\begin{align}\label{eq(d53)}
&W_{n}\left[\frac{d}{dy}\left(\frac{e^{xh^{-1}(y)}}
{l(h^{-1}(y))l(y)}\right)\right]\Bigg{|}_{y=0}\notag \\
&\qquad =W_{n}\left[\left(x\frac{1}{h'(h^{-1}(y))}-\frac{l'(y)}{l(y)}
-\frac{1}{h'(h^{-1}(y))}\frac{l'(h^{-1}(y))}{l(h^{-1}(y))}\right)
\frac{e^{xh^{-1}(y)}}{l(h^{-1}(y))l(y)}\right]\Bigg{|}_{y=0},
\end{align}
which, by using Property III, yields
\begin{align*}
&W_{n}\left[\frac{d}{dy}\left(\frac{e^{xh^{-1}(y)}}
{l(h^{-1}(y))l(y)}\right)\right]\Bigg{|}_{y=0}\notag \\
&\qquad =P_{n}\left[x\frac{1}{h'(h^{-1}(y))}-\frac{l'(y)}{l(y)}
-\frac{1}{h'(h^{-1}(y))}\frac{l'(h^{-1}(y))}{l(h^{-1}(y))}\right]
\Bigg{|}_{y=0}W_{n}\left[\frac{e^{xh^{-1}(y)}}{l(h^{-1}(y))l(y)}
\right]\Bigg{|}_{y=0}
\end{align*}
or, equivalently,
\begin{align}\label{eq(d54)}
&W_{n}\left[\frac{d}{dy}\left(\frac{e^{xh^{-1}(y)}}
{l(h^{-1}(y))l(y)}\right)\right]\Bigg{|}_{y=0}\notag \\
&\qquad =
\begin{bmatrix}
xa_{0}+b_0+c_0 & 0 & 0 & \cdots & 0 \\[3pt]
xa_{1}+b_1+c_1 & xa_{0}+b_0+c_0 & 0 & \cdots & 0 \\[3pt]
xa_{2}+b_2+c_2 & \binom{2}{1}(xa_{1}+b_1+c_1) & xa_{0}+b_0+c_0 & \cdots & 0 \\[3pt]
\vdots & \vdots & \vdots & \ddots & \vdots \\[3pt]
xa_{n}+b_n+c_n & \binom{n}{1}(xa_{n-1}+b_{n-1}+c_{n-1})
&\cdots &\cdots & xa_{0}+b_0+c_0 \\
\end{bmatrix}
\notag \\
&\qquad \qquad \qquad \qquad \cdot
\begin{bmatrix}
{_s}A_{0}(x)  \\
{_s}A_{1}(x)  \\
{_s}A_{2}(x)  \\
\vdots  \\
{_s}A_{n-1}(x)  \\
{_s}A_{n}(x)  \\
\end{bmatrix}.
\end{align}
Equating the $n${th} rows of \eqref{eq(d52)} and \eqref{eq(d54)},
we arrive at desired result \eqref{eq(d50)} asserted by Theorem \ref{thm3.3}.
\end{proof}

In its special case when $l(y)=1$ and $b_{k}=c_{k}=0$ $(\forall\; k \geqq 1)$,
the recursive formula \eqref{eq(d50)} asserted by Theorem \ref{thm3.3},
we obtain Corollary \ref{cor3.3} below.

\begin{corollary}\label{cor3.3}
Let ${_q}A_{n}(x)\thicksim(1,h(y))$ be the associated polynomial sequence.
Then
\begin{eqnarray}\label{eq(d55)}
{_q}A_{n+1}(x)=x\sum_{k=0}^{n}\binom{n}{k} a_k ~{_q}A_{n-k}(x),
\end{eqnarray}
where
$$a_{k}=\left(\frac{1}{h'(h^{-1}(y))}\right)^{(k)}\Bigg{|}_{y=0}.$$
\end{corollary}

\section{Concluding remarks and observations}
In the preceding sections, we have developed a differential equation
and recurrence relations for the Sheffer-Appell polynomials by using the
Pascal functional and Wronskian matrices. In order to derive these recursive
formulas for the Sheffer-Appell polynomials,
we find several interesting recurrence relations
for such related polynomials as (for example) the
generalized Laguerre polynomials $L_{n}^{(\lambda)}(x)$
and the Miller-Lee type Appell polynomials $G_{n}^{(m)}(x)$.
The results presented in this article
are potentially useful in deducing further interesting formulas
for other specific classes of orthogonal polynomials.


\begin{thebibliography}{12}

\bibitem{Aceto}
Aceto L. and Ca\c{c}$\tilde{a}$o I.; A matrix approach to Sheffer polynomials,
\emph{J. Math. Anal. Appl.} {\bf 446} (2017), 87--100.

\bibitem{Andrews}
Andrews G.E., Askey R., Roy R.; \emph{Special Functions},
Cambridge University Press, Cambridge, 1999.

\bibitem{Bucchianico}
Bucchianico A. di  and Loeb D.; A selected survey of umbral calculus,
\emph{Electron. J. Combin.} {\bf 2} (2000), 1--34.

\bibitem{Dattoli}
Dattoli G., Migliorati M. and Srivastava H.M.; Sheffer polynomials, monomiality
principle, \break algebraic methods and the theory of classical polynomials,
{\it Math. Comput. Model.} {\bf 45} (2007), 1033--1041.

\bibitem{Dere}
Dere R., Simsek Y. and Srivastava H.M.; A unified
presentation of three families of generalized Apostol
type polynomials based upon the theory of the
umbral calculus and the umbral algebra,
{\it J. Number Theory} {\bf 133} (2013), 3245--3263.

\bibitem{He}
He M.-X. and Ricci P.E.; Differential equation of Appell polynomials via
the factorization method, \emph{J. Comput. Appl. Math.}
{\bf 139} (2002), 231--237.

\bibitem{Subuhi}
Khan S. and Riyasat M.; A determinantal approach to Sheffer-Appell
polynomials via monomiality principle, \emph{J. Math. Anal. Appl.}
{\bf 421} (2015), 806--829.

\bibitem{Lehmer}
Lehmer D.H.; A new approach to Bernoulli polynomials,
\emph{Amer. Math. Monthly} {\bf 95} (1988), 905--911.

\bibitem{Pinter}
Pint\'{e}r \'{A}. and Srivastava H.M.; Addition theorems for
the Appell polynomials and the associated
classes of polynomial expansions,
{\it Aequationes Math.} {\bf 85} (2013), 483--495.

\bibitem{Rosen}
Rosen K.; \emph{Handbook of Discrete and Combinatorial Mathematics},
CRC Press, Boca Raton, Florida, 2000.

\bibitem{Roman1}
Roman S.; The theory of the umbral calculus. I,
\emph{J. Math. Anal. Appl.}
{\bf 87} (1982), 58--115.

\bibitem{Roman2}
Roman S.; \emph{The Umbral Calculus}, Academic Press, New York, 1984.

\bibitem{Roman3}
Roman S. and Rota G.-C.; The umbral calculus, \emph{Adv. Math.}
{\bf 27} (1978), 95--188.

\bibitem{Sheffer1}
Sheffer I.M.; A differential equation for Appell polynomials,
\emph{Bull. Amer. Math. Soc.} {\bf 41} (1935), 914--923.

\bibitem{Sheffer2}
Sheffer I.M.; Some properties of polynomial sets of type zero,
\emph{Duke Math. J.} {\bf 5} (1939), 590--622.

\bibitem{HMS1982}
Srivastava H.M.; Some characterizations of Appell and
$q$-Appell polynomials, {\it Ann.
Mat. Pura Appl.} ({\it Ser.} 4) {\bf 130} (1982), 321--329.

\bibitem{HMS2014}
Srivastava H.M., Nisar K.S. and Khan M.A.;
Some umbral calculus presentations of the
Chan-Chyan-Srivastava polynomials and the
Erku\c{s}-Srivastava polynomials,
{\it Proyecciones J. Math.} {\bf 33} (2014), 77--90.

\bibitem{Yang1}
Yang Y.-Z. and Micek C.; Generalized Pascal functional
matrix and its applications,
\emph{Linear Algebra Appl.} {\bf 423} (2007), 230--245.

\bibitem{Yang2}
Yang Y.-Z. and Youn H.-Y.; Appell polynomial sequences:
A linear algebra approach,
\emph{JP J. Algebra Number Theory Appl.}
{\bf 13} 2009, 65--98.

\bibitem{Youn}
Youn H.-Y. and Yang Y.-Z.;
Differential equation and recursive formulas of
Sheffer polynomial sequences, \emph{ISRN Discrete Math.}
{\bf 2011} (2011) Article ID 476462, 1--16.

\end{thebibliography}
\end{document}